\documentclass[11pt,reqno]{amsart}
\usepackage{bbm}
\usepackage{mathrsfs}
\usepackage{color,xcolor}

\usepackage{amsfonts}
\allowdisplaybreaks
\usepackage{amsmath,amsthm,amsfonts,amssymb,latexsym,epsfig}
\usepackage{color}
\usepackage [latin1]{inputenc}
\usepackage{cite}

\newtheorem{theorem}{Theorem}
\newtheorem{thm}{Theorem}
\newtheorem{lemma}{Lemma}

\newtheorem*{lemA}{Lemma A}
\newtheorem*{lemB}{Lemma B}

\newtheorem{corollary}{Corollary}

\newtheorem{definition}{Definition}
\newtheorem{remark}{Remark}

\numberwithin{equation}{section}

\numberwithin{theorem}{section}
\numberwithin{lemma}{section}
\numberwithin{corollary}{section}

\usepackage{ifpdf}
\ifpdf
  \usepackage[hyperindex,pagebackref]{hyperref}%
\else
  \expandafter\ifx\csname dvipdfm\endcsname\relax
    \usepackage[hypertex,hyperindex,pagebackref]{hyperref}
  \else
    \usepackage[dvipdfm,hyperindex,pagebackref]{hyperref}
  \fi
\fi

\theoremstyle{remark}

\makeatletter
\@namedef{subjclassname@2020}{\textup{2020} Mathematics Subject Classification}
\makeatother

\date{today}
\begin{document}
\setcounter{page}{1}

\title[Bohr and Landau-Type Theorems in Harmonic Mappings]
{Bohr, Bohr-Rogosinski, and Landau-Type Results for a Generalized Class of Harmonic Mappings}

\author[X. Y. Wang, X. T. Han, R. Hossain, M. B. Ahamed]{Xiaoyuan Wang, Xintong Han, Rajesh Hossain, and Molla Basir Ahamed$^*$}

\vskip.10in
\address{\noindent Xiaoyuan Wang \vskip.05in
 School of Mathematical Sciences, Liaocheng University, Liaocheng
	252059, Shandong, People's Republic of China
}
\email{\textcolor[rgb]{0.00,0.00,0.84}{mewangxiaoyuan$@$163.com}}

\address{\noindent Xintong Han\vskip.05in
School of Mathematics and Statistics, Nanjing University of Science and Technology,
 Nanjing 210094, Jiangsu, P. R. China.}\vskip.05in
\email{\textcolor[rgb]{0.00,0.00,0.84}{1059207043@qq.com}}

\address{\noindent Rajesh Hossain\vskip.05in
	Department of Mathematics, Jadavpur University, Kolkata-700032, West Bengal, India.}\vskip.05in
\email{\textcolor[rgb]{0.00,0.00,0.84}{rajesh1998hossain@gmail.com}}

\address{\noindent Molla Basir Ahamed\vskip.05in
	Department of Mathematics, Jadavpur University, Kolkata-700032, West Bengal, India.}\vskip.05in
\email{\textcolor[rgb]{0.00,0.00,0.84}{mbahamed.math@jadavpuruniversity.in}}

\thanks{$^*$Corresponding author.}



\date{\today}

\begin{abstract}
	In this paper, we study the Bohr phenomenon for a generalized subclass of harmonic mappings defined by a second-order differential inequality in the unit disk. Specifically, we consider the class $\mathcal{BH}_0(\gamma, \delta)$, which extends several known subclasses of harmonic and analytic functions. By employing sharp coefficient estimates and growth results, we establish improved versions of Bohr-type inequalities, including refined Bohr radii and Bohr--Rogosinski radii for this class. Furthermore, we derive generalized inequalities involving higher-order coefficient sums and area terms, thereby extending classical Bohr inequalities in a harmonic setting. The sharpness of the obtained results is verified through extremal functions. In addition, we obtain Landau-type theorems for the class $\mathcal{BH}_0(\gamma, \delta)$, providing explicit bounds for the radius of univalence and the size of schlicht disks contained in the image domain. Our results not only unify and extend several earlier works but also provide new insights into the geometric behavior of harmonic mappings under differential constraints.
\end{abstract}
\maketitle

\section{Introduction}
The primary objective of this paper is to establish several Bohr-type inequalities for harmonic functions. Let $\mathbb{D}=\{z \in \mathbb{C}: |z|<1\}$ denote the open unit disk. A harmonic mapping in $\mathbb{D}$ is a complex-valued function $f = u + iv$ of $z = x + iy$ that satisfies the Laplace equation $\Delta f = 4 f_{z\bar{z}} = 0$. 

Every harmonic mapping $f$ in $\mathbb{D}$ admits the canonical decomposition $f = h + \overline{g}$, where $h$ and $g$ are analytic functions in $\mathbb{D}$. The Jacobian of $f$ is given by
\[
J_f = |h'(z)|^2 - |g'(z)|^2.
\]
By Lewy's theorem (see \cite{Duren2004}), the mapping $f$ is locally univalent and sense-preserving in $\mathbb{D}$ if and only if $J_f(z) > 0$ in $\mathbb{D}$. This condition is equivalent to requiring that $h'(z) \neq 0$ in $\mathbb{D}$ and that the dilatation $\omega = g'/h'$ satisfies $|\omega(z)| < 1$ in $\mathbb{D}$.\vspace{1.2mm}

Let $\mathcal{H}_0$ denote the class of complex-valued harmonic mappings $f = h + \overline{g}$ in $\mathbb{D}$, normalized by
\[
h(0) = g(0) = 0, \quad h'(0) = 1, \quad g'(0) = 0.
\]
Every function $f \in \mathcal{H}_0$ admits the canonical representation $f = h + \overline{g}$, where
\begin{equation}\label{e-1.1}
	h(z) = z + \sum_{n=2}^{\infty} a_n z^n \quad \text{and} \quad g(z) = \sum_{n=2}^{\infty} b_n z^n.
\end{equation}
The theory of harmonic functions has been widely investigated, and among its various aspects, the Bohr phenomenon has attracted considerable attention. This phenomenon has been studied in several settings and generalizations, including multidimensional cases \cite{Aizenberg2000, Liu2021} and complete Reinhardt domains \cite{Boas1997}.

We begin by recalling the classical definition of the Bohr radius. Let $\mathcal{B}$ denote the class of analytic functions $f$ in $\mathbb{D}$ such that $|f(z)| < 1$ for all $z \in \mathbb{D}$. In $1914$, Bohr \cite{Bohr1914} proved that if $f \in \mathcal{B}$ is of the form
\[
f(z) = \sum_{n=0}^{\infty} a_n z^n,
\]
then the majorant series $M_f(r) = \sum_{n=0}^{\infty} |a_n| r^n$ satisfies
\begin{equation}\label{e-1.2}
	M_{f_0}(r) = \sum_{n=1}^{\infty} |a_n| r^n \leq 1 - |a_0| = d(f(0), \partial f(\mathbb{D}))
\end{equation}
for all $|z| = r \leq 1/6$, where $f_0(z) = f(z) - f(0)$. 

Subsequently, Wiener, Riesz, and Schur independently improved inequality \eqref{e-1.2} by showing that it holds for $|z| \leq 1/3$, and that $1/3$ is the sharp Bohr radius for functions in $\mathcal{B}$. Inequality \eqref{e-1.2} is known as the Bohr inequality, and the associated phenomenon is referred to as the Bohr phenomenon.

We now introduce several variants of the Bohr radius. The first type, which incorporates the area term $S_r$, was studied by Kayumov and Ponnusamy as follows.
\begin{thm}\label{thm-A}
 Let $f(z)=\sum_{n=0}^{\infty} a_{n} z^{n}$ be analytic in $\mathbb{D},|f(z)| \leq 1$ and $S_{r}$ denote the area of the image $|z|<r$ under the mapping $f$. Then
$$
B_{1}(r):=\sum_{n=0}^{\infty}\left|a_{n}\right| r^{n}+\frac{16}{9}\left(\frac{S_{r}}{\pi}\right) \leq 1 \text { for } r \leq \frac{1}{3},
$$
and the values $1 / 3$ and $16 / 9$ cannot be improved. Moreover,
$$
B_{2}(r):=\left|a_{0}\right|^{2}+\sum_{n=1}^{\infty}\left|a_{n}\right| r^{n}+\textcolor{black}{\frac{9}{8}}\left(\frac{S_{r}}{\pi}\right) \leq 1 \text { for } r \leq \frac{1}{2},
$$
and the values $1 / 2$ and \textcolor{black}{$9 / 8$} cannot be improved.
\end{thm}
Another variant, known as the Bohr-Rogosinski radius, was introduced by Kayumov and Ponnusamy in  and is defined as follows.
\begin{thm}\label{thm-B}
 Suppose that $f(z)=\sum_{n=0}^{\infty} a_{n} z^{n}$ is analytic in $\mathbb{D}$ and $|f(z)|<1$ in $\mathbb{D}$. Then
$$
|f(z)|+\sum_{n=N}^{\infty}\left|a_{n}\right| r^{n} \leq 1 \text { for } r \leq R_{N},
$$
where $R_{N}$ is the positive root of the equation
$$
2(1+r) r^{N}-\left(1-r^{2}\right)=0.
$$
The radius $R_{N}$ is the best possible. Moreover,
$$
|f(z)|^{2}+\sum_{n=N}^{\infty}\left|a_{n}\right| r^{n} \leq 1 \text { for } r \leq R_{N}^{\prime},
$$
where $R_{N}^{\prime}$ is the positive root of the equation
$$
(1+r) r^{N}-\left(1-r^{2}\right)=0 .
$$
The radius $R_{N}^{\prime}$ is the best possible.
\end{thm}
The Bohr--Rogosinski radius for the class of bounded analytic functions $\mathcal{B}$ is characterized as follows: for any $f \in \mathcal{B}$ and $N \geq 1$, the partial sums
\[
S_N(z) = \sum_{n=0}^{N} a_n z^n
\]
satisfy $|S_N(z)| < 1$ in the disk $\mathbb{D}_{1/2}$. The radius $1/2$ is sharp. Accordingly, the Bohr--Rogosinski sum $R_N^f(z)$ associated with $f$ is defined by
\begin{equation}\label{e-1.3}
	R_N^f(z) := |f(z)| + \sum_{n=N}^{\infty} |a_n| r^n, \quad (|z| = r).
\end{equation}
Observe that, for $N=1$, the expression in \eqref{e-1.3} reduces to the classical Bohr sum, where $f(0)$ is replaced by $|f(z)|$.

The class of harmonic mappings $f$ given by \eqref{e-1.1} can be regarded as a natural extension of analytic functions. Consequently, the study of properties of harmonic mappings analogous to those of analytic functions has become an important topic in geometric function theory. In particular, the investigation of Bohr-type radii for harmonic mappings was initiated in \cite{AbuMuhanna2010, AbuMuhanna2014, Kayumov2018a}. Since then, a number of authors have studied Bohr-type phenomena for various subclasses of harmonic mappings (see, for example, \cite{Ahamed2021a, Ahamed2023, Allu2022, Allu2021a, Evdoridis2019, Gangania2022,Ahamed2024b,AhammedAhamed2024a,AhammedAhamed2024b,Ahamed2024a,AhamedAllu2022,AhamedAllu2023}). \vspace{1.2mm}

We now recall the definitions of specific geometric classes of functions and investigate their properties throughout this paper. Li and Ponnusamy \cite{Li2016}  studied various geometric properties of the class
\begin{align*}
	\widetilde{\mathcal{G}}_{\mathcal{H}}^{0}(\beta)=\left\{f=h+\bar{g} \in \mathcal{H}_{0}: \operatorname{Re}\left(\frac{h(z)}{z}\right)-\beta>\left|\frac{g(z)}{z}\right|, z \in \mathbb{D}\right\}.
\end{align*}
In \cite{Ghosh2019}, Ghosh and Vasudevarao  have recently studied the following class $\mathcal{W}_{\mathcal{H}}^{0}(\alpha)$ finding sharp coefficient estimates and growth theorems, where
\begin{align*}
	\mathcal{W}_{\mathcal{H}}^{0}(\alpha):=\left\{f=h+\bar{g} \in \mathcal{H}_{0}: \operatorname{Re}\left(h^{\prime}(z)+\alpha z h^{\prime \prime}(z)\right)>\left|g^{\prime}(z)+\alpha z g^{\prime \prime}(z)\right|, z \in \mathbb{D}\right\}.
\end{align*}
Liu and Yang \cite{Liu2019}  introduced the subclass $\mathcal{G}_{\mathcal{H}}^{k}(\alpha)$ which is defined by
$$
\begin{aligned}
\mathcal{G}_{\mathcal{H}}^{k}(\alpha):=\Big\{f=h+\bar{g} \in \mathcal{H}_{0}^{k}: \operatorname{Re} \Big( &   (1-\alpha) \frac{h(z)}{z}+\alpha h^{\prime}(z)\Big ) \\
&  >\Big|(1-\alpha) \frac{g(z)}{z}+\alpha g^{\prime}(z)\Big|, z \in \mathbb{D}\Big\},
\end{aligned}
$$
where $\mathcal{H}_{0}^{k}=\left\{f=h+\bar{g} \in \mathcal{H}: h^{\prime}(0)-1=g^{\prime}(0)=h^{\prime \prime}(0)=\cdots=h^{(k)}(0)=g^{(k)}(0)=0\right\}$ and $\mathcal{H}_{0}^{1} \equiv \mathcal{H}_{0}$ for some $\alpha \geq 0, k \geq 1$.\vspace{1.2mm}

Allu and Halder \cite{Allu2021b} investigated the Bohr phenomenon for the subclasses $\widetilde{\mathcal{G}}_{\mathcal{H}}^{0}(\beta)$, $\mathcal{W}_{\mathcal{H}}^{0}(\alpha)$, and $\mathcal{G}_{\mathcal{H}}^{k}(\alpha)$. In addition, Bohr-type inequalities have been studied in various other settings, including the quasi-subordination family \cite{Ponnusamy2022}, log-harmonic mappings \cite{Ali2016}, and quasiconformal mappings \cite{Kayumov2018b}. 

More recently, Yal\c{c}{\i}n et al. \cite{Yalcin2024} introduced and investigated a new subclass of harmonic mappings. In the present paper, we establish several Bohr-type inequalities for this subclass and demonstrate the sharpness of the obtained results.
\begin{definition}
 Denote by $\mathcal{B} \mathcal{H}^{0}(\gamma, \delta)$ the class of functions $f=h+\bar{g} \in \textcolor{black}{\mathcal{H}_{0}}$ and satisfy
$$
\operatorname{Re}\left[\gamma \frac{h(z)}{z}+\delta h^{\prime}(z)+\left(\frac{\delta-\gamma}{2}\right) z h^{\prime \prime}(z)\right]>\left|\gamma \frac{g(z)}{z}+ \delta g^{\prime}(z)+\left(\frac{\delta-\gamma}{2}\right) z g^{\prime \prime}(z)\right|,
$$
where $\delta \geq \gamma \geq 0$ and $z \in \mathbb{D}$.
\end{definition}
The subclass $\mathcal{BH}^0(\gamma, \delta)$ can be viewed as a generalization of the harmonic subclasses $\mathcal{W}_{\mathcal{H}}^{0}(\alpha)$ and $\mathcal{W}_{\mathcal{H}}^{0}(1)$ studied by Nagpal and Ravichandran \cite{Nagpal2014}. Moreover, it extends the analytic case corresponding to $g(z)=0$, as discussed in \cite{AlRefai2019}. It is worth noting that $\mathcal{BH}^0(\gamma, \delta)$ forms a subclass of close-to-convex harmonic mappings defined via suitable coefficient conditions (see also \cite{Cakmak-Yasar-Yalcin-2022, Meher-Gochhayat-2024} for related classes). Thus, the class $\mathcal{BH}^0(\gamma, \delta)$, depending on the two parameters $\gamma$ and $\delta$, provides a natural and flexible generalization within the theory of harmonic mappings.

In the present paper, we establish several Bohr-type inequalities for functions belonging to the class $\mathcal{BH}^0(\gamma, \delta)$.

To obtain our main results, we require the following auxiliary lemmas. The first lemma (Lemma~\ref{lem-1.1}) provides sharp coefficient estimates for functions in the class $\mathcal{BH}^0(\gamma, \delta)$, while the second lemma (Lemma~\ref{lem-1.2}) gives a growth theorem that enables the estimation of the Euclidean distance $d(f(0), \partial f(\mathbb{D}))$.
\begin{lemma}\label{lem-1.1}
 (\cite{Yalcin2024}) Let $f \in \mathcal{B H}^{0}(\gamma, \delta)$. Then, for $n \geq 2$,
 \begin{enumerate}
 	\item[(i)] $\left|a_{n}\right|+\left|b_{n}\right| \leq \dfrac{4(\delta+\gamma)}{(n+1)[n(\delta-\gamma)+2 \gamma]}$;\vspace{2mm}\\
 	
 	\item[(ii)] $\left|\left|a_{n}\right|-\left|b_{n}\right|\right| \leq \dfrac{4(\delta+\gamma)}{(n+1)[n(\delta-\gamma)+2 \gamma]}$;\vspace{2mm}\\
 	
 	\item[(iii)] $\left|a_{n}\right| \leq \dfrac{4(\delta+\gamma)}{(n+1)[n(\delta-\gamma)+2 \gamma]}$.
 \end{enumerate}
All bounds are sharp. Equality holds for all bounds   if
\begin{equation}\label{e-1.4}
f(z)=z+\sum_{n=2}^{\infty} \frac{4(\delta+\gamma)}{(n+1)[n(\delta-\gamma)+2 \gamma]} z^{n}. 
\end{equation}
\end{lemma}
\begin{lemma}\label{lem-1.2}
 (\cite{Yalcin2024}) Let $f \in \mathcal{B} \mathcal{H}^{0}(\gamma, \delta)$. Then,
\begin{equation}\label{e-1.401}
|z|+\sum_{n=2}^{\infty} \frac{(-1)^{n-1} 4(\delta+\gamma)}{(n+1)[n(\delta-\gamma)+2 \gamma]}|z|^{n} \leq|f(z)| \leq|z|+\sum_{n=2}^{\infty} \frac{4(\delta+\gamma)}{(n+1)[n(\delta-\gamma)+2 \gamma]}|z|^{n}
\end{equation}
hold. The result is sharp and equalities apply to the function
$$
f(z)=z+\sum_{n=2}^{\infty} \frac{4(\delta+\gamma)}{(n+1)[n(\delta-\gamma)+2 \gamma]} z^{n}.
$$
\end{lemma}
\begin{remark}
\textcolor{black}{ In fact, \eqref{e-1.401} is a modified version.
The specific left-hand inequality in   original  \cite[Theorem 5]{Yalcin2024} says that
$$
|f(z)| \geq \max \left\{0, r-4(\delta+\gamma) \sum_{n=2}^{\infty} \frac{r^n}{(n+1)[n(\delta-\gamma)+2 \gamma]}\right\}
$$
for any $f=h+\bar{g} \in \mathcal{B} \mathcal{H}^0(\gamma, \delta)$ and $|z|=r<1$.
However, the real-valued lower bound in \eqref{e-1.401} can be  derived from the foundational results, 
specifically  \cite[Theorem 5]{Yalcin2024}  (the growth theorem) and the coefficient bounds from   \cite[Theorem 3]{Yalcin2024}, namely,  Lemma \ref{lem-1.1}. }

\end{remark}

\begin{lemma}\label{lem-1.3} (\cite{Yalcin2024}) Let $f$ be a function of type $f=h+\bar{g}$ in $\mathcal{B} \mathcal{H}^{0}(\gamma, \delta)$. Then
$$
\left|b_{n}\right| \leq \frac{2(\delta+\gamma)}{(n+1)[n(\delta-\gamma)+2 \gamma]}.
$$
The result is sharp and equality applies to the function
\begin{equation}\label{e-1.5}
f(z)=z+\frac{2(\delta+\gamma)}{(n+1)[n(\delta-\gamma)+2 \gamma]} \bar{z}^{n}. 
\end{equation}
 for each $n \geq 2$.
\end{lemma}
Despite significant progress in the study of Bohr-type phenomena for various subclasses of analytic and harmonic mappings, the corresponding theory for more general parameter-dependent harmonic classes remains incomplete. In particular, for the recently introduced class $\mathcal{BH}^0(\gamma, \delta)$ with $\delta \geq \gamma \geq 0$, a systematic treatment of Bohr-type inequalities, including improved and refined Bohr radii as well as Bohr--Rogosinski type phenomena, has not been thoroughly investigated. Moreover, the interplay between such Bohr-type estimates and geometric properties, such as Landau-type results for this class, has not yet been explored. 

The primary objective of this paper is to investigate the Bohr phenomenon for the class $\mathcal{BH}^0(\gamma, \delta)$ with $\delta \geq \gamma \geq 0$. More precisely, we establish improved and refined Bohr radii, derive Bohr--Rogosinski type inequalities, and obtain generalized Bohr inequalities involving higher-order coefficient sums and area terms. In addition, we establish a Landau-type theorem for this class, providing explicit bounds for the radius of univalence and the size of schlicht disks contained in the image domain.

The results obtained in this paper not only extend several known results for analytic and harmonic mappings but also unify and generalize earlier findings in the literature. The inclusion of Landau-type estimates further highlights the geometric significance of the class $\mathcal{BH}^0(\gamma, \delta)$ and reveals new connections between coefficient bounds, growth estimates, and geometric function theory. 

The rest of the paper is organized as follows. In Section~\ref{sec-2}, we present the main Bohr-type results, including improved, refined, and Bohr--Rogosinski inequalities. In Section~\ref{Sec-3}, we establish Landau-type theorems for functions in the class $\mathcal{BH}^0(\gamma, \delta)$ and discuss their geometric implications.
\section{\bf Bohr inequality and Bohr-Rogosinski inequality for functions in the class $\mathcal{BH}^0(\gamma, \delta)$}\label{sec-2}
We begin with the following result, which establishes an improved Bohr-type inequality for functions in the class $\mathcal{BH}^0(\gamma,\delta)$ and determines the corresponding sharp radius.
\begin{theorem}\label{thm-2.1}
 Let $f \in \mathcal{B} \mathcal{H}^{0}(\gamma, \delta)$. Then, when $n \geq 2$ and $p \geq 1$,
\begin{equation}\label{e-2.1}
|z|+\sum_{n=2}^{\infty}\left(\left|a_{n}\right|+\left|b_{n}\right|\right)|z|^{n}+\sum_{n=2}^{\infty}\left(\left|a_{n}\right|+\left|b_{n}\right|\right)^{p}|z|^{p n} \leq d(f(0), \partial f(\mathbb{D})) 
\end{equation}
holds for $|z|=r \leq r_{p}(\gamma, \delta)$, where $r_{p}(\gamma, \delta)$ is the unique root of
$$
\begin{aligned}
& r+\sum_{n=2}^{\infty} \frac{4(\delta+\gamma) r^{n}}{(n+1)[n(\delta-\gamma)+2 \gamma]}+\sum_{n=2}^{\infty}\left(\frac{4(\delta+\gamma) r^{n}}{(n+1)[n(\delta-\gamma)+2 \gamma]}\right)^{p} \\
& \quad=1+\sum_{n=2}^{\infty} \frac{(-1)^{n-1} 4(\delta+\gamma)}{(n+1)[n(\delta-\gamma)+2 \gamma]}.
\end{aligned}
$$
Here $r_{p}(\gamma, \delta)$ is the best possible.
\end{theorem}
Our results generalize several known results and yield new consequences. In particular, when $\mathcal{BH}^0(0,1) := \mathcal{W}_{\mathcal{H}}^{0}(1/2)$, our findings reduce to \cite[Corollary 2.4]{Ahamed2022}. Moreover, for the specific choice $\gamma = 1$, $\delta = 3$, and $p = 2$, we obtain the following corollary of Theorem~\ref{thm-2.1}.

\begin{corollary}\label{cor-2.1}
Let $f \in \mathcal{B H}^{0}(1,3)$  be of the form \eqref{e-1.1}. Then
$$
|z|+\sum_{n=2}^{\infty}\left(\left|a_{n}\right|+\left|b_{n}\right|\right)|z|^{n}+\sum_{n=2}^{\infty}\left(\left|a_{n}\right|+\left|b_{n}\right|\right)^{2}|z|^{2 n} \leq d(f(0), \partial f(\mathbb{D}))
$$
holds for $|z|=r \leq r_{2}(1,3) \approx 0.309260$, where $r_{2}(1,3)$ is the unique root of the equation
$$
r+\sum_{n=2}^{\infty} \frac{8 r^{n}}{(n+1)^{2}}+\sum_{n=2}^{\infty} \frac{64 r^{2 n}}{(n+1)^{4}}-1-\frac{2}{3}\left(9-\pi^{2}\right)=0
$$
in $(0,1)$. Here $r_{2}(1,3)$ is the best possible.
\end{corollary}

\begin{corollary}\label{cor-2.2}
 Let $f \in \mathcal{B} \mathcal{H}^{0}(1,3)$ be of the form \eqref{e-1.1}. Then
$$
|z|+\sum_{n=2}^{\infty}\left(\left|a_{n}\right|+\left|b_{n}\right|\right)^{2}|z|^{2 n} \leq d(f(0), \partial f(\mathbb{D}))
$$
holds for $|z|=r \leq r_{2}^{*}(1,3) \approx 0.399130$, where $r_{2}^{*}(1,3)$ is the unique root of the equation
$$
r+\sum_{n=2}^{\infty} \frac{64 r^{2 n}}{(n+1)^{4}}-1-\frac{2}{3}\left(9-\pi^{2}\right)=0
$$
in $(0,1)$. Here $r_{2}^{*}(1,3)$ is the best possible.
\end{corollary}
Next, we present a Bohr--Rogosinski type inequality for the class $\mathcal{BH}^0(\gamma,\delta)$, along with the identification of the associated sharp radius.
\begin{theorem}\label{thm-2.3}
 Let $f \in \mathcal{B} \mathcal{H}^{0}(\gamma, \delta)$. Then for integers $m \geq 1, N \geq 2$, the inequality
\begin{equation}\label{e-2.3}
\left|f\left(z^{m}\right)\right|+\sum_{n=N}^{\infty}\left(\left|a_{n}\right|+\left|b_{n}\right|\right)|z|^{n} \leq d(f(0), \partial f(\mathbb{D})) 
\end{equation}
holds for $|z|=r \leq R_{m, N}(\gamma, \delta)$, where $R_{m, N}(\gamma, \delta)$ is the unique root of
$$
\begin{aligned}
r^{m}+ & \sum_{n=2}^{\infty} \frac{4(\delta+\gamma) r^{m n}}{(n+1)[n(\delta-\gamma)+2 \gamma]}+\sum_{n=N}^{\infty} \frac{4(\delta+\gamma) r^{n}}{(n+1)[n(\delta-\gamma)+2 \gamma]} \\
& -1-\sum_{n=2}^{\infty} \frac{(-1)^{n-1} 4(\delta+\gamma)}{(n+1)[n(\delta-\gamma)+2 \gamma]}=0.
\end{aligned}
$$
Here $R_{m, N}(\gamma, \delta)$ is the best possible.
\end{theorem}

For a particular choice of $\gamma, \delta$, we have the following corollary of Theorem \ref{thm-2.3}.

\begin{corollary}\label{thm-cor-2.2}
Let $f \in \mathcal{B} \mathcal{H}^{0}(\gamma, \delta)$. Then, when $n \geq 2$,
\begin{equation}\label{e-2.2}
|f(z)|+\sum_{n=2}^{\infty}\left(\left|a_{n}\right|+\left|b_{n}\right|\right)|z|^{n} \leq d(f(0), \partial f(\mathbb{D})) 
\end{equation}
holds for $|z|=r \leq r_{*}(\gamma, \delta)$, where $r_{*}(\gamma, \delta)$ is the unique root of
$$
r+\sum_{n=2}^{\infty} \frac{8(\delta+\gamma) r^{n}}{(n+1)[n(\delta-\gamma)+2 \gamma]}-1-\sum_{n=2}^{\infty} \frac{(-1)^{n-1} 4(\delta+\gamma)}{(n+1)[n(\delta-\gamma)+2 \gamma]}=0.
$$
Here $r_{*}(\gamma, \delta)$ is the best possible.
\end{corollary}

\begin{corollary}\label{cor-2.3}
Let $f \in \mathcal{B} \mathcal{H}^{0}(1,3)$ be of the form \eqref{e-1.1}. Then
$$
|f(z)|+\sum_{n=2}^{\infty}\left(\left|a_{n}\right|+\left|b_{n}\right|\right)|z|^{n} \leq d(f(0), \partial f(\mathbb{D}))
$$
holds for $|z|=r \leq r^{*}(1,3) \approx 0.313516$, where $r^{*}(1,3)$ is the unique root of the equation
$$
r+\sum_{n=2}^{\infty} \frac{16 r^{n}}{(n+1)^{2}}-1-\frac{2}{3}\left(9-\pi^{2}\right)=0
$$
in $(0,1)$. Here $r^{*}(1,3)$ is the best possible.
\end{corollary}
The following theorem extends the classical Bohr inequality by incorporating area terms and provides generalized sharp bounds for functions in $\mathcal{BH}^0(\gamma,\delta)$.
\begin{theorem}\label{thm-2.4}
Let $f \in \mathcal{B} \mathcal{H}^{0}(\gamma, \delta)$. Then\\
(i)
$$
|z|+\sum_{n=2}^{\infty}\left(\left|a_{n}\right|+\left|b_{n}\right|\right)|z|^{n}+ P_N\left(\frac{S_r}{\pi}\right)   \leq d(f(0), \partial f(\mathbb{D}))
$$
holds for $|z|=r \leq r_{f}(\gamma, \delta)$, where $P_N\left(\omega\right)=\lambda_N\omega^N+\cdots+\lambda_1\omega,\;\lambda_j\geq 0,\; \mbox{for}\; j=1, 2, \ldots, N, $ and   $r_{f}(\gamma, \delta)$ is the unique root of the equation
$$
\begin{aligned}
r & +\sum_{n=2}^{\infty} \frac{4(\delta+\gamma) r^{n}}{(n+1)[n(\delta-\gamma)+2 \gamma]}+  P_N\left( r^{2}+  \sum_{n=2}^{\infty} \frac{16 n(\delta+\gamma)^{2} r^{2 n}}{\{(n+1)[n(\delta-\gamma)+2 \gamma]\}^{2}}\right) \\
& -1-\sum_{n=2}^{\infty} \frac{(-1)^{n-1} 4(\delta+\gamma)}{(n+1)[n(\delta-\gamma)+2 \gamma]}=0
\end{aligned}
$$
in $(0,1)$. Here $r_{f}(\gamma, \delta)$ is the best possible.\\
(ii)
$$
|f(z)|^{p}+\sum_{n=2}^{\infty}\left(\left|a_{n}\right|+\left|b_{n}\right|\right)|z|^{n}+\left(\frac{S_{r}}{\pi}\right)^{2} \leq d(f(0), \partial f(\mathbb{D}))
$$
holds for $|z|=r \leq r_{f}^{*}(\gamma, \delta)$, where $r_{f}^{*}(\gamma, \delta)$ is the unique root of the equation
$$
\begin{aligned}
& \left(r+\sum_{n=2}^{\infty} \frac{4(\delta+\gamma) r^{n}}{(n+1)[n(\delta-\gamma)+2 \gamma]}\right)^{p}+\sum_{n=2}^{\infty} \frac{4(\delta+\gamma) r^{n}}{(n+1)[n(\delta-\gamma)+2 \gamma]} \\
& \quad+\left(r^{2}+\sum_{n=2}^{\infty} \frac{16n(\delta+\gamma)^{2} r^{2 n}}{\{(n+1)[n(\delta-\gamma)+2 \gamma]\}^{2}}\right)^{2}-1-\sum_{n=2}^{\infty} \frac{(-1)^{n-1} 4(\delta+\gamma)}{(n+1)[n(\delta-\gamma)+2 \gamma]}=0
\end{aligned}
$$
in $(0,1)$. Here $r_{f}^{*}(\gamma, \delta)$ is the best possible.
\end{theorem}

For a particular choice of $\gamma, \delta$, we have the following corollary of Theorem \ref{thm-2.4}.
\begin{corollary}\label{cor-2.4}
Let $f \in \mathcal{B} \mathcal{H}^{0}(1,3)$ be given by \eqref{e-1.1}. Then\\
(i)
$$
|z|+\sum_{n=2}^{\infty}\left(\left|a_{n}\right|+\left|b_{n}\right|\right)|z|^{n}+\frac{S_{r}}{\pi} \leq d(f(0), \partial f(\mathbb{D}))
$$
holds for $|z|=r \leq \dot{r}(1,3) \approx 0.268346$, where $\dot{r}(1,3)$ is the unique root of
$$
r^{2}+r+\sum_{n=2}^{\infty} \frac{8 r^{n}}{(n+1)^{2}}+\sum_{n=2}^{\infty} \frac{64 r^{2 n}}{(n+1)^{4}}-1-\frac{2}{3}\left(9-\pi^{2}\right)=0
$$
in $(0,1)$. Here $\dot{r}(1,3)$ is the best possible.\\
(ii)
$$
|f(z)|^{2}+\sum_{n=2}^{\infty}\left(\left|a_{n}\right|+\left|b_{n}\right|\right)|z|^{n}+\frac{S_{r}}{\pi} \leq d(f(0), \partial f(\mathbb{D}))
$$
holds for $|z|=r \leq \ddot{r}(1,3) \approx 0.359414$, where $\ddot{r}(1,3)$ is the unique root of
$$
\left(r+\sum_{n=2}^{\infty} \frac{8 r^{n}}{(n+1)^{2}}\right)^{2}+\sum_{n=2}^{\infty} \frac{8 r^{n}}{(n+1)^{2}}+\left(r^{2}+\sum_{n=2}^{\infty} \frac{64 r^{2 n}}{(n+1)^{4}}\right)^{2}-1-\frac{2}{3}\left(9-\pi^{2}\right)=0
$$
in $(0,1)$. Here $\ddot{r}(1,3)$ is the best possible.
\end{corollary}
Finally, we obtain refined Bohr-type inequalities involving the analytic and co-analytic parts separately, together with the corresponding sharp radii.
\begin{theorem}\label{thm-2.5}
 Let $f \in \mathcal{B} \mathcal{H}^{0}(\gamma, \delta)$. Then\\
(i) for $|z|=r \leq r_{h}(\gamma, \delta)$,
$$
|z|+P_N(|h(z)|)++\sum_{n=2}^{\infty}\left|a_{n}\right||z|^{n} \leq d(f(0), \partial f(\mathbb{D}))
$$
holds, where $r_{h}(\gamma, \delta)$ is the unique root of the equation
\begin{align*}
 &r+\sum_{n=2}^{\infty} \frac{4(\delta+\gamma) r^{n}}{(n+1)[n(\delta-\gamma)+2 \gamma]} + P_N \left(r+ \sum_{n=2}^{\infty} \frac{4(\delta+\gamma) r^{n}}{(n+1)[n(\delta-\gamma)+2 \gamma]}\right)\\
 &\quad  -1-\sum_{n=2}^{\infty} \frac{(-1)^{n-1} 4(\delta+\gamma)}{(n+1)[n(\delta-\gamma)+2 \gamma]}=0
\end{align*}
in $(0,1)$. Here $r_{h}(\gamma, \delta)$ is the best possible.\\
(ii) for $|z|=r \leq r_{g}(\gamma, \delta)$,
$$
|z|+P_N(|g(z)|)+\sum_{n=2}^{\infty}\left|b_{n}\right||z|^{n} \leq d(f(0), \partial f(\mathbb{D}))
$$
holds, where $r_{g}(\gamma, \delta)$ is the unique root of the equation
\begin{align*}
& r+\sum_{n=2}^{\infty} \frac{2(\delta+\gamma) r^{n}}{(n+1)[n(\delta-\gamma)+2 \gamma]} + P_N\left(\sum_{n=2}^{\infty} \frac{2(\delta+\gamma) r^{n}}{(n+1)[n(\delta-\gamma)+2 \gamma]}\right)\\
 &\quad  -1-\sum_{n=2}^{\infty} \frac{(-1)^{n-1} 4(\delta+\gamma)}{(n+1)[n(\delta-\gamma)+2 \gamma]}=0
\end{align*}
in (0,1). Here $r_{g}(\gamma, \delta)$ is the best possible.
\end{theorem}

\begin{corollary}\label{cor-2.5}
 Let $f \in \mathcal{B} \mathcal{H}^{0}(1,3)$ be given by \eqref{e-1.1}. Then
$$
|g(z)|+\sum_{n=2}^{\infty}\left|b_{n}\right||z|^{n} \leq d(f(0), \partial f(\mathbb{D}))
$$
holds for $|z|=r \leq \widetilde{r}(1,3) \approx 0.560537$, where $\widetilde{r}(1,3)$ is the unique root of the equation
$$
\sum_{n=2}^{\infty} \frac{4 r^{n}}{(n+1)^{2}}-1-\frac{2}{3}\left(9-\pi^{2}\right)=0
$$
in $(0,1)$. Here $\widetilde{r}(1,3)$ is the best possible.
\end{corollary}
We now discuss the proof of the results of this section.
\begin{proof}[Proof of  Theorem \ref{thm-2.1}]
Let $f \in \mathcal{B} \mathcal{H}^{0}(\gamma, \delta)$. Then, using Lemmas \ref{lem-1.1}  and \ref{lem-1.2} , it is easy to see that $d(f(0), \partial f(\mathbb{D}))$ satisfies
\begin{equation}\label{e-3.1}
d(f(0), \partial f(\mathbb{D}))=\liminf _{|z| \rightarrow 1}|f(z)-f(0)| \geq 1+\sum_{n=2}^{\infty} \frac{(-1)^{n-1} 4(\delta+\gamma)}{(n+1)[n(\delta-\gamma)+2 \gamma]}. 
\end{equation}
Let $\Phi_{1}:[0,1] \rightarrow \mathbb{R}$ be given as
\begin{align}\label{e-3.2}
\Phi_{1}(r)= & r+\sum_{n=2}^{\infty} \frac{4(\delta+\gamma) r^{n}}{(n+1)[n(\delta-\gamma)+2 \gamma]}+\sum_{n=2}^{\infty}\left(\frac{4(\delta+\gamma) r^{n}}{(n+1)[n(\delta-\gamma)+2 \gamma]}\right)^{p}  
\notag \\
& -1-\sum_{n=2}^{\infty} \frac{(-1)^{n-1} 4(\delta+\gamma)}{(n+1)[n(\delta-\gamma)+2 \gamma]}.
\end{align}
We see that the function $\Phi_{1}(r)$ is continuous on $[0,1]$ and differentiable on $(0,1)$. Since $n \geq 2$, we have
$$
\left|\sum_{n=2}^{\infty} \frac{(-1)^{n-1} 4(\delta+\gamma)}{(n+1)[n(\delta-\gamma)+2 \gamma]}\right| \leq 1.
$$

Indeed,  let $k=  {2 \gamma}/{(\delta-\gamma)}, \delta>\gamma\ge 0$, then
\begin{equation}\label{e-3.1+1}
 S(k)=\sum_{n=2}^{\infty} b_n(k)
\end{equation} is an interleaved series,
where $$b_n(k)=\frac{4(1+k)}{(n+1)(n+k)} $$ is regarding $n$ monotonically decreasing towards zero. So it meets the Leibniz discriminant criteria.

\noindent When $0\le k\le 2$, $$|S(k)|\le b_2(k) \leq b_2(2)=1.$$
When $ k>  2, $ (also includes  $\delta=\gamma> 0$ ) consider $S_2(k)=-b_2+b_3$ and  $S_3(k)=-b_2+b_3-b_4$, then
 $S_2(k)\leq S(k)\leq S_3(k)$.  By calculation,
 $$ |S(k)|\leq \max \{|S_2(k)|,|S_3(k)|\}= \max \{|16-24\ln2|,2-4\ln2\}=4\ln2-2<1.$$

It follows that
$$
\Phi_{1}(0)=-1-\sum_{n=2}^{\infty} \frac{(-1)^{n-1} 4(\delta+\gamma)}{(n+1)[n(\delta-\gamma)+2 \gamma]} \leq 0 .
$$
On the other hand, since
$$
\sum_{n=2}^{\infty} \frac{(\delta+\gamma)}{(n+1)[n(\delta-\gamma)+2 \gamma]} \geq \sum_{n=2}^{\infty} \frac{(-1)^{n-1}(\delta+\gamma)}{(n+1)[n(\delta-\gamma)+2 \gamma]} \quad \text { for } n \geq 2
$$
by computation, we have
$$
\begin{aligned}
\Phi_{1}(1)= & \sum_{n=2}^{\infty} \frac{4(\delta+\gamma)}{(n+1)[n(\delta-\gamma)+2 \gamma]}+\sum_{n=2}^{\infty} \frac{4^{p}(\delta+\gamma)^{p}}{\{(n+1)[n(\delta-\gamma)+2 \gamma]\}^{p}} \\
& -\sum_{n=2}^{\infty} \frac{(-1)^{n-1} 4(\delta+\gamma)}{(n+1)[n(\delta-\gamma)+2 \gamma]} \\
\geq & \sum_{n=2}^{\infty} \frac{4^{p}(\delta+\gamma)^{p}}{\{(n+1)[n(\delta-\gamma)+2 \gamma]\}^{p}}>0.
\end{aligned}
$$
Hence, $\Phi_{1}(0) \Phi_{1}(1)<0$, and by the Intermediate Value Theorem, we see that $\Phi_{1}(r)$ has roots in $(0,1)$. Next, we show that $\Phi_{1}(r)$ has exactly one root in $(0,1)$. It is sufficient to show that $\Phi_{1}$ is a strictly monotonic function in $(0,1)$. By computation, it follows that
$$
\Phi_{1}^{\prime}(r)=1+\sum_{n=2}^{\infty} \frac{4 n(\delta+\gamma) r^{n-1}}{(n+1)[n(\delta-\gamma)+2 \gamma]}+\sum_{n=2}^{\infty} \frac{4^{p} n p(\delta+\gamma)^{n p} r^{n p-1}}{\{(n+1)[n(\delta-\gamma)+2 \gamma]\}^{p}}>0
$$
for all $r \in(0,1)$, which shows that $\Phi_{1}$ is a strictly increasing function. Thus, $\Phi_{1}(r)$ has the unique root in $(0,1)$, written $r_{p}(\gamma, \delta)$. Hence, we have $\Phi_{1}\left(r_{p}(\gamma, \delta)\right)=0$, and thus from (3.2), we obtain
\begin{align}\label{e-3.3}
& r_{p}(\gamma, \delta)+\sum_{n=2}^{\infty} \frac{4(\delta+\gamma) r_{p}(\gamma, \delta)}{(n+1)[n(\delta-\gamma)+2 \gamma]}+\sum_{n=2}^{\infty} \frac{4^{p} n p(\delta+\gamma)^{n p} r_{p}^{n p-1}(\gamma, \delta)}{\{(n+1)[n(\delta-\gamma)+2 \gamma]\}^{p}} \notag \\   
& \quad=1+\sum_{n=2}^{\infty} \frac{(-1)^{n-1} 4(\delta+\gamma)}{(n+1)[n(\delta-\gamma)+2 \gamma]}.
\end{align}
To show that $r_{p}(\gamma, \delta)$ is the best possible, we consider the function $f=f_{\gamma, \delta}$ given by \eqref{e-1.4}. We see that the function $f_{\gamma, \delta} \in \mathcal{B} \mathcal{H}^{0}(\gamma, \delta)$ and for $f=f_{\gamma, \delta}$, it yields
\begin{equation}\label{e-3.4}
d(f(0), \partial f(\mathbb{D}))=1+\sum_{n=2}^{\infty} \frac{(-1)^{n-1} 4(\delta+\gamma)}{(n+1)[n(\delta-\gamma)+2 \gamma]}. 
\end{equation}
It is also obvious that
$$
\begin{aligned}
r & +\sum_{n=2}^{\infty} \frac{4(\delta+\gamma) r^{n}}{(n+1)[n(\delta-\gamma)+2 \gamma]}+\sum_{n=2}^{\infty}\left(\frac{4(\delta+\gamma) r^{n}}{(n+1)[n(\delta-\gamma)+2 \gamma]}\right)^{p} \\
& \leq 1+\sum_{n=2}^{\infty} \frac{(-1)^{n-1} 4(\delta+\gamma)}{(n+1)[n(\delta-\gamma)+2 \gamma]}
\end{aligned}
$$
for $r \leq r_{p}(\gamma, \delta)$. By using \eqref{e-3.3} and \eqref{e-3.4} for $f=f_{\gamma, \delta}$ and $r=r_{p}(\gamma, \delta)$, it follows
$$
\begin{aligned}
& |z|+\sum_{n=2}^{\infty}\left(\left|a_{n}\right|+\left|b_{n}\right|\right)|z|^{n}+\sum_{n=2}^{\infty}\left(\left|a_{n}\right|+\left|b_{n}\right|\right)^{p}|z|^{p n} \\
& =r+\sum_{n=2}^{\infty} \frac{4(\delta+\gamma) r^{n}}{(n+1)[n(\delta-\gamma)+2 \gamma]}+\sum_{n=2}^{\infty}\left(\frac{4(\delta+\gamma) r^{n}}{(n+1)[n(\delta-\gamma)+2 \gamma]}\right)^{p} \\
& =r_{p}(\gamma, \delta)+\sum_{n=2}^{\infty} \frac{4(\delta+\gamma) r_{p}^{n}(\gamma, \delta)}{(n+1)[n(\delta-\gamma)+2 \gamma]}+\sum_{n=2}^{\infty}\left(\frac{4(\delta+\gamma) r_{p}^{n}(\gamma, \delta)}{(n+1)[n(\delta-\gamma)+2 \gamma]}\right)^{p} \\
& =1+\sum_{n=2}^{\infty} \frac{(-1)^{n-1} 4(\delta+\gamma)}{(n+1)[n(\delta-\gamma)+2 \gamma]} \\
& =d(f(0), \partial f(\mathbb{D})).
\end{aligned}
$$
Therefore, $r_{p}(\gamma, \delta)$ is the best possible. The proof is completed.
\end{proof}

\begin{proof}[Proof of  Theorem \ref{thm-2.3}]  Let $f \in \mathcal{B} \mathcal{H}^{0}(\gamma, \delta)$. Then, in view of Lemmas  \ref{lem-1.1}, \ref{lem-1.2} and \eqref{e-3.1}, for $|z|=r$, we have
\begin{align}\label{e-3.8}
& \left|f\left(z^{m}\right)\right|+\sum_{n=N}^{\infty}\left(\left|a_{n}\right|+\left|b_{n}\right|\right)|z|^{n} \notag \\
& \leq|z|^{m}+\sum_{n=2}^{\infty} \frac{4(\delta+\gamma)|z|^{m n}}{(n+1)[n(\delta-\gamma)+2 \gamma]}+\sum_{n=2}^{\infty} \frac{4(\delta+\gamma)|z|^{n}}{(n+1)[n(\delta-\gamma)+2 \gamma]}. 
\end{align}
Through computation, it shows that
$$
\begin{aligned}
r^{m} & +\sum_{n=2}^{\infty} \frac{4(\delta+\gamma) r^{m n}}{(n+1)[n(\delta-\gamma)+2 \gamma]}+\sum_{n=2}^{\infty} \frac{4(\delta+\gamma) r^{n}}{(n+1)[n(\delta-\gamma)+2 \gamma]} \\
& \leq 1+\sum_{n=2}^{\infty} \frac{(-1)^{n-1} 4(\delta+\gamma)}{(n+1)[n(\delta-\gamma)+2 \gamma]}
\end{aligned}
$$
for $r \leq R_{m, N}(\gamma, \delta)$, where $R_{m, N}(\gamma, \delta)$ is a root of $\Phi_{3}(r)=0$ in $(0,1)$, where $\Phi_{3}:[0,1] \rightarrow \mathbb{R}$ is given by
$$
\begin{aligned}
\Phi_{3}(r):=r^{m} & +\sum_{n=2}^{\infty} \frac{4(\delta+\gamma) r^{m n}}{(n+1)[n(\delta-\gamma)+2 \gamma]}+\sum_{n=2}^{\infty} \frac{4(\delta+\gamma) r^{n}}{(n+1)[n(\delta-\gamma)+2 \gamma]} \\
& -1-\sum_{n=2}^{\infty} \frac{(-1)^{n-1} 4(\delta+\gamma)}{(n+1)[n(\delta-\gamma)+2 \gamma]}.
\end{aligned}
$$
Computation shows that $\Phi_{3}(0) \Phi_{3}(1)<0$. Also, $\Phi_{3}(r)$ is strictly increasing, indeed,
$$
\Phi_{3}^{\prime}(r)=m r^{m-1}+\sum_{n=2}^{\infty} \frac{4 m n(\delta+\gamma) r^{m n-1}}{(n+1)[n(\delta-\gamma)+2 \gamma]}+\sum_{n=2}^{\infty} \frac{4 n(\delta+\gamma) r^{n-1}}{(n+1)[n(\delta-\gamma)+2 \gamma]}>0 .
$$
Also, the function $\Phi_{3}$ is differentiable on $(0,1)$ and by the Intermediate Value Theorem, we see that $\Phi_{3}(r)$ has a unique root in $(0,1)$, written by $R_{m, N}(\gamma, \delta)$. Therefore, we derive that $\Phi_{3}\left(R_{m, N}(\gamma, \delta)\right)=0$, which is equivalent to
\begin{align}\label{e-3.9}
& R_{m, N}^{m} +\sum_{n=2}^{\infty} \frac{4(\delta+\gamma) R_{m, N}^{m n}}{(n+1)[n(\delta-\gamma)+2 \gamma]}+\sum_{n=2}^{\infty} \frac{4(\delta+\gamma) R_{m, N}^{n}}{(n+1)[n(\delta-\gamma)+2 \gamma]} \notag \\
& \quad=1+\sum_{n=2}^{\infty}  \frac{(-1)^{n-1}4(\delta+\gamma)}{(n+1)[n(\delta-\gamma)+2\gamma]}. 
\end{align}
The process to show that $R_{m, N}(\gamma, \delta)$ is the best possible is similar to Theorem \ref{thm-2.1}.
\textcolor{black}{Consider the extremal function $f=f_{(\gamma, \delta)}$ given by \eqref{e-1.4},
which belongs to $\mathcal{BH}^0(\gamma, \delta)$ and satisfies:
\[
d(f(0), \partial f(\mathbb{D})) = 1 + \sum_{n=2}^{\infty} \frac{(-1)^{n-1}4(\delta+\gamma)}{(n+1)[n(\delta-\gamma)+2\gamma]}.
\]
 In view of \eqref{e-3.8} and \eqref{e-3.9}, for $f=f_{(\gamma, \delta)}$ and $|z|=R_{m, N}$, we obtain
$$
\begin{aligned}
 \left|f_{(\gamma, \delta)}\left(z^m\right)\right|+\sum_{n=N}^{\infty}\left(\left|a_n\right|+\left|b_n\right|\right)|z|^n& =R_{m, N}^{m} +\sum_{n=2}^{\infty} \frac{4(\delta+\gamma) R_{m, N}^{m n}}{(n+1)[n(\delta-\gamma)+2 \gamma]} \\
& \quad+\sum_{n=2}^{\infty} \frac{4(\delta+\gamma) R_{m, N}^{n}}{(n+1)[n(\delta-\gamma)+2 \gamma]}  \\
& =1+\sum_{n=2}^{\infty}  \frac{(-1)^{n-1}4(\delta+\gamma)}{(n+1)[n(\delta-\gamma)+2\gamma]} \\
& =d\left(f_{(\gamma, \delta)}(0), \partial f_{(\gamma, \delta)}(\mathbb{D})\right)
\end{aligned}
$$
This shows that $R_{m, N}$ is the best possible.} This completes the proof.
\end{proof}

\begin{proof}[Proof of Theorem \ref{thm-2.4}]
 Let $f \in \mathcal{B} \mathcal{H}^{0}(\gamma, \delta)$ and the subdisk $\mathbb{D}_{r}:=\{z \in \mathbb{C}:|z|<r\}$. By [14, p.113], we have the area of the harmonic mapping $f$ is
\begin{equation}\label{e-3.10}
S_{r}=\iint_{\mathbb{D}_{r}} J_{f}(z) d x d y=\iint_{\mathbb{D}_{r}}\left(\left|h^{\prime}(z)\right|^{2}-\left|g^{\prime}(z)\right|^{2}\right) d x d y 
\end{equation}
By calculation,
\begin{align}\label{e-3.11}
\iint_{\mathbb{D}_{r}}\left|h^{\prime}(z)\right|^{2} d x d y & =\int_{0}^{r} \int_{0}^{2 \pi}\left|h^{\prime}\left(\rho e^{i \theta}\right)\right|^{2} \rho d \theta d \rho \notag \\
& =\int_{0}^{r} \int_{0}^{2 \pi} \rho\left(\sum_{n=1}^{\infty} n a_{n} \rho^{n-1} e^{i(n-1) \theta}\right)\left(\sum_{n=1}^{\infty} n \bar{a}_{n} \rho^{n-1} e^{-i(n-1) \theta}\right) d \theta d \rho \notag \\
& =\int_{0}^{r}\left(\sum_{n=1}^{\infty} 2 \pi n^{2}\left|a_{n}\right|^{2} \rho^{2 n-1}\right) d \rho \notag \\
& =\pi \sum_{n=1}^{\infty} n\left|a_{n}\right|^{2} r^{2 n}. 
\end{align}
Similarly,
\begin{equation}\label{e-3.12}
\iint_{\mathbb{D}_{r}}\left|g^{\prime}(z)\right|^{2} d x d y=\pi \sum_{n=2}^{\infty} n\left|b_{n}\right|^{2} r^{2 n} 
\end{equation}
Considering Lemma \ref{lem-1.1}  along with equations \eqref{e-3.10}, \eqref{e-3.11} and \eqref{e-3.12}, we derive
\begin{align}\label{e-3.13}
\frac{S_{r}}{\pi} & =\frac{1}{\pi} \iint_{\mathbb{D}_{r}}\left(\left|h^{\prime}(z)\right|^{2}-\left|g^{\prime}(z)\right|^{2}\right) d x d y \notag \\
& =r^{2}+\sum_{n=2}^{\infty} n\left(\left|a_{n}\right|+\left|b_{n}\right|\right)\left(\left|a_{n}\right|-\left|b_{n}\right|\right) r^{2 n} \notag \\
& =r^{2}+\sum_{n=2}^{\infty} \frac{16 n(\delta+\gamma)^{2} r^{2 n}}{\{(n+1)[n(\delta-\gamma)+2 \gamma]\}^{2}}. 
\end{align}

(i) Using Lemma \ref{lem-1.1} with \eqref{e-3.13}   for $|z|=r$, we get
\begin{align}\label{e-3.14}
& |z|+\sum_{n=2}^{\infty}\left(\left|a_{n}\right|+\left|b_{n}\right|\right)|z|^{n}+ P_N\left(\frac{S_r}{\pi}\right)   \notag \\
& \leq r+\sum_{n=2}^{\infty} \frac{4(\delta+\gamma) r^{n}}{(n+1)[n(\delta-\gamma)+2 \gamma]}+ P_N\left( r^{2}+\sum_{n=2}^{\infty} \frac{16 n(\delta+\gamma)^{2} r^{2 n}}{\{(n+1)[n(\delta-\gamma)+2 \gamma]\}^{2}}\right). 
\end{align}
It is evident that
$$
\begin{aligned}
 &  r+\sum_{n=2}^{\infty} \frac{4(\delta+\gamma) r^{n}}{(n+1)[n(\delta-\gamma)+2 \gamma]}+ P_N\left( r^{2} +\sum_{n=2}^{\infty} \frac{16 n(\delta+\gamma)^{2} r^{2 n}}{\{(n+1)[n(\delta-\gamma)+2 \gamma]\}^{2}}\right)   \\
& \leq 1+\sum_{n=2}^{\infty} \frac{(-1)^{n-1} 4(\delta+\gamma)}{(n+1)[n(\delta-\gamma)+2 \gamma]}
\end{aligned}
$$
for $r \leq r_{f}(\gamma, \delta)$, where $r_{f}(\gamma, \delta)$ is the root of $\Phi_{4}(r)=0$, where $\Phi_{4}:[0,1] \rightarrow \mathbb{R}$ given by
$$
\begin{aligned}
\Phi_{4}(r):= & r+\sum_{n=2}^{\infty} \frac{4(\delta+\gamma) r^{n}}{(n+1)[n(\delta-\gamma)+2 \gamma]}+P_N\left( r^{2}+ \sum_{n=2}^{\infty} \frac{16 n(\delta+\gamma)^{2} r^{2 n}}{\{(n+1)[n(\delta-\gamma)+2 \gamma]\}^{2}}\right)  \\
& -1-\sum_{n=2}^{\infty} \frac{(-1)^{n-1} 4(\delta+\gamma)}{(n+1)[n(\delta-\gamma)+2 \gamma]}.
\end{aligned}
$$
It is easy to show that $\Phi_{4}(0) \Phi_{4}(1)<0$ and $\Phi_{4}^{\prime}(r)>0$ for $r \in(0,1)$. Then, by the Intermediate Value Theorem, $\Phi_{4}$ has a unique root in $(0,1)$, denoted by $r_{f}(\gamma, \delta)$. Therefore, we derive
\begin{align}\label{e-3.15}
& r_{f}(\gamma, \delta)+\sum_{n=2}^{\infty} \frac{4(\delta+\gamma) r_{f}^{n}(\gamma, \delta)}{(n+1)[n(\delta-\gamma)+2 \gamma]}+P_N\left( r_{f}^{2}(\gamma, \delta)+ \sum_{n=2}^{\infty} \frac{16 n(\delta+\gamma)^{2} r_{f}^{2 n}(\gamma, \delta)}{\{(n+1)[n(\delta-\gamma)+2 \gamma]\}^{2}}\right) \notag \\
& =1+\sum_{n=2}^{\infty} \frac{(-1)^{n-1} 4(\delta+\gamma)}{(n+1)[n(\delta-\gamma)+2 \gamma]}. 
\end{align}
To show that $r_{f}(\gamma, \delta)$ is the best possible, we consider the function $f=f_{(\gamma, \delta)}$ given by equation \eqref{e-1.4}. Combining equations  \eqref{e-3.4}, \eqref{e-3.14}  and \eqref{e-3.15}, for $f=f_{(\gamma, \delta)}$ and $r= r_{f}(\gamma, \delta)$, it yields
$$
\begin{aligned}
& |z|+\sum_{n=2}^{\infty}\left(\left|a_{n}\right|+\left|b_{n}\right|\right)|z|^{n}+P_N\left(\frac{S_{r_{f}(\gamma, \delta)}}{\pi}\right)  \\
& =r_{f}(\gamma, \delta)+\sum_{n=2}^{\infty} \frac{4(\delta+\gamma) r_{f}^{n}(\gamma, \delta)}{(n+1)[n(\delta-\gamma)+2 \gamma]}+P_N\left(r_{f}^{2}(\gamma, \delta)+ \sum_{n=2}^{\infty} \frac{16 n(\delta+\gamma)^{2} r_{f}^{2 n}(\gamma, \delta)}{\{(n+1)[n(\delta-\gamma)+2 \gamma]\}^{2}}\right)  \\
& =1+\sum_{n=2}^{\infty} \frac{(-1)^{n-1} 4(\delta+\gamma)}{(n+1)[n(\delta-\gamma)+2 \gamma]} \\
& =d\left(f_{(\gamma, \delta)}(0), \partial f_{(\alpha, \beta, \gamma)}(\mathbb{D})\right).
\end{aligned}
$$
This shows that the radius $r_{f}(\gamma, \delta)$ is the best possible. This completes the proof of (i).

 (ii) In view of Lemmas \ref{lem-1.2} and \ref{lem-1.3}, and \eqref{e-3.13}  for $|z|=r$, we obtain
\begin{align}\label{e-3.16}
& |f(z)|^{p}+\sum_{n=2}^{\infty}\left(\left|a_{n}\right|+\left|b_{n}\right|\right)|z|^{n}+\left(\frac{S_{r}}{\pi}\right)^{2} \notag \\
& \quad \leq\left(|z|+\sum_{n=2}^{\infty}\left(\left|a_{n}\right|+\left|b_{n}\right|\right)|z|^{n}\right)^{p}+\sum_{n=2}^{\infty}\left(\left|a_{n}\right|+\left|b_{n}\right|\right)|z|^{n}+\left(\frac{S_{r}}{\pi}\right)^{2} \notag \\
& \quad \leq\left(r+\sum_{n=2}^{\infty} \frac{4(\delta+\gamma) r^{n}}{(n+1)[n(\delta-\gamma)+2 \gamma]}\right)^{p}+\sum_{n=2}^{\infty} \frac{4(\delta+\gamma) r^{n}}{(n+1)[n(\delta-\gamma)+2 \gamma]} \notag \\
& \quad+\left(r^{2}+\sum_{n=2}^{\infty} \frac{16 n(\delta+\gamma)^{2} r^{2 n}}{\{(n+1)[n(\delta-\gamma)+2 \gamma]\}^{2}}\right)^{2}. 
\end{align}
A simple computation shows that
$$
\begin{aligned}
& \left(r_{f}^{*}(\gamma, \delta)+\sum_{n=2}^{\infty} \frac{4(\delta+\gamma)\left(r_{f}^{*}(\gamma, \delta)\right)^{n}}{(n+1)[n(\delta-\gamma)+2 \gamma]}\right)^{p}+\sum_{n=2}^{\infty} \frac{4(\delta+\gamma)\left(r_{f}^{*}(\gamma, \delta)\right)^{n}}{(n+1)[n(\delta-\gamma)+2 \gamma]} \\
& \quad+\left(\left(r_{f}^{*}(\gamma, \delta)\right)^{2}+\sum_{n=2}^{\infty} \frac{16 n(\delta+\gamma)^{2}\left(r_{f}^{*}(\gamma, \delta)\right)^{2 n}}{\{(n+1)[n(\delta-\gamma)+2 \gamma]\}^{2}}\right)^{2} \leq 1+\sum_{n=2}^{\infty} \frac{(-1)^{n-1} 4(\delta+\gamma)}{(n+1)[n(\delta-\gamma)+2 \gamma]}
\end{aligned}
$$
for $r \leq r_{f}^{*}(\gamma, \delta)$, where $r_{f}^{*}(\gamma, \delta)$ is a root of $\Phi_{5}(r)=0$ in $(0,1)$, where $\Phi_{5}:[0,1] \rightarrow \mathbb{R}$ is defined by
$$
\begin{aligned}
\Phi_{5}(r)& = \left(r_{f}^{*}(\gamma, \delta)+\sum_{n=2}^{\infty} \frac{4(\delta+\gamma)\left(r_{f}^{*}(\gamma, \delta)\right)^{n}}{(n+1)[n(\delta-\gamma)+2 \gamma]}\right)^{p}+\sum_{n=2}^{\infty} \frac{4(\delta+\gamma)\left(r_{f}^{*}(\gamma, \delta)\right)^{n}}{(n+1)[n(\delta-\gamma)+2 \gamma]} \\
& +\left(\left(r_{f}^{*}(\gamma, \delta)\right)^{2}+\sum_{n=2}^{\infty} \frac{16 n(\delta+\gamma)^{2}\left(r_{f}^{*}(\gamma, \delta)\right)^{2 n}}{\{(n+1)[n(\delta-\gamma)+2 \gamma]\}^{2}}\right)^{2}-1-\sum_{n=2}^{\infty} \frac{(-1)^{n-1} 4(\delta+\gamma)}{(n+1)[n(\delta-\gamma)+2 \gamma]}.
\end{aligned}
$$
Using the same argument as in the proof of Theorem \ref{thm-2.1}, we can easily show that $\Phi_{5}(r)=0$ has the unique root, denoted as $r_{f}^{*}(\gamma, \delta)$. The process to show the best possible of $r_{f}^{*}(\gamma, \delta)$ is similar to (i). This concludes the proof of (ii).
\end{proof}

\begin{proof}[Proof of Theorem \ref{thm-2.5}]
 (i) For $f \in \mathcal{B} \mathcal{H}^{0}(\gamma, \delta),|z|=r$, Lemma \ref{lem-1.1} shows that
\begin{equation}\label{e-3.17}
|z|+P_N(|h(z)|)+\sum_{n=2}^{\infty}\left|a_{n}\right||z|^{n} \leq d(f(0), \partial f(\mathbb{D})) 
\end{equation}
if
\begin{align*}
  r&+\sum_{n=2}^{\infty} \frac{4(\delta+\gamma) r^{n}}{(n+1)[n(\delta-\gamma)+2 \gamma]} + P_N \left(r+ \sum_{n=2}^{\infty} \frac{4(\delta+\gamma) r^{n}}{(n+1)[n(\delta-\gamma)+2 \gamma]}\right)\\
   &\quad \leq 1+\sum_{n=2}^{\infty} \frac{(-1)^{n-1} 4(\delta+\gamma)}{(n+1)[n(\delta-\gamma)+2 \gamma]}.
\end{align*}
Denoted $\Phi_{6}:[0,1] \rightarrow \mathbb{R}$ by
\begin{align*}
\Phi_{6}(r)&:= r+\sum_{n=2}^{\infty} \frac{4(\delta+\gamma) r^{n}}{(n+1)[n(\delta-\gamma)+2 \gamma]} + P_N \left(r+ \sum_{n=2}^{\infty} \frac{4(\delta+\gamma) r^{n}}{(n+1)[n(\delta-\gamma)+2 \gamma]}\right)\\
&\quad -1-\sum_{n=2}^{\infty} \frac{(-1)^{n-1} 4(\delta+\gamma)}{(n+1)[n(\delta-\gamma)+2 \gamma]}.
\end{align*}
It can be shown that $\Phi_{6}$ has a unique root in $(0,1)$, denoted by $r_{h}(\gamma, \delta)$, satisfying
\begin{align}\label{e-3.18}
r_{h}(\gamma, \delta)& +\sum_{n=2}^{\infty} \frac{4(\delta+\gamma)\left(r_{h}(\gamma, \delta)\right)^{n}}{(n+1)[n(\delta-\gamma)+2 \gamma]}   + P_N \left(r_{h}(\gamma, \delta)+\sum_{n=2}^{\infty} \frac{4(\delta+\gamma)\left(r_{h}(\gamma, \delta)\right)^{n}}{(n+1)[n(\delta-\gamma)+2 \gamma]}\right)\notag \\
& =1+\sum_{n=2}^{\infty} \frac{(-1)^{n-1} 4(\delta+\gamma)}{(n+1)[n(\delta-\gamma)+2 \gamma]}. 
\end{align}
Now to prove that $r_{h}(\gamma, \delta)$ is best possible, we consider $f_{(\gamma, \delta)}=h_{(\gamma, \delta)}+g_{(\gamma, \delta)}$ given by \eqref{e-1.4}. Combining \eqref{e-3.4}, \eqref{e-3.17} and \eqref{e-3.18}, for $f=f_{(\gamma, \delta)}$ and $|z|=r_{h}(\gamma, \delta)$, it yields
$$
\begin{aligned}
|z|&+P_N(\left|h_{(\gamma, \delta)}(z)\right|)+\sum_{n=2}^{\infty}\left|a_{n}\right||z|^{n}\\ & = r_{h}(\gamma, \delta)+\sum_{n=2}^{\infty} \frac{4(\delta+\gamma)\left(r_{h}(\gamma, \delta)\right)^{n}}{(n+1)[n(\delta-\gamma)+2 \gamma]} + P_N \left( r_{h}(\gamma, \delta)+\sum_{n=2}^{\infty} \frac{4(\delta+\gamma)\left(r_{h}(\gamma, \delta)\right)^{n}}{(n+1)[n(\delta-\gamma)+2 \gamma]}\right)\\
& =1+\sum_{n=2}^{\infty} \frac{(-1)^{n-1} 4(\delta+\gamma)}{(n+1)[n(\delta-\gamma)+2 \gamma]} \\
& =d\left(f_{(\gamma, \delta)}(0), \partial f_{(\gamma, \delta)}(\mathbb{D})\right).
\end{aligned}
$$
Hence, $r_{h}(\gamma, \delta)$ is best possible.\\
(ii) Define $\Phi_{7}:[0,1] \rightarrow \mathbb{R}$ by
\begin{align*}
\Phi_{7}(r)&:=r+\sum_{n=2}^{\infty} \frac{2(\delta+\gamma) r^{n}}{(n+1)[n(\delta-\gamma)+2 \gamma]}+  P_N \left(\sum_{n=2}^{\infty} \frac{2(\delta+\gamma) r^{n}}{(n+1)[n(\delta-\gamma)+2 \gamma]}\right)\\ 
& -1-\sum_{n=2}^{\infty} \frac{(-1)^{n-1} 4(\delta+\gamma)}{(n+1)[n(\delta-\gamma)+2 \gamma]}.
\end{align*}
$\Phi_{7}(r)=0$ has a unique root in $(0,1)$, denoted by $r_{g}(\gamma, \delta)$, satisfying
\begin{align}\label{e-3.19}
r_{g}(\gamma, \delta)&+\sum_{n=2}^{\infty} \frac{2(\delta+\gamma)\left(r_{g}(\gamma, \delta)\right)^{n}}{(n+1)[n(\delta-\gamma)+2 \gamma]}+ P_N \left(\sum_{n=2}^{\infty} \frac{2(\delta+\gamma)\left(r_{g}(\gamma, \delta)\right)^{n}}{(n+1)[n(\delta-\gamma)+2 \gamma]}\right)\notag \\
& =1+\sum_{n=2}^{\infty} \frac{(-1)^{n-1} 4(\delta+\gamma)}{(n+1)[n(\delta-\gamma)+2 \gamma]}. 
\end{align}
For $|z|=r$, by Lemma \ref{lem-1.3}, it implies
\begin{align}\label{e-3.20}
|z|+P_N(|g(z)|)+\sum_{n=2}^{\infty}\left|b_{n}\right||z|^{n} \leq d(f(0), \partial f(\mathbb{D})), \end{align}
if
\begin{align*}
r&+\sum_{n=2}^{\infty} \frac{2(\delta+\gamma) r^{n}}{(n+1)[n(\delta-\gamma)+2 \gamma]} +  P_N \left(\sum_{n=2}^{\infty} \frac{2(\delta+\gamma) r^{n}}{(n+1)[n(\delta-\gamma)+2 \gamma]}\right)
\notag \\
& \leq 1+\sum_{n=2}^{\infty} \frac{(-1)^{n-1} 4(\delta+\gamma)}{(n+1)[n(\delta-\gamma)+2 \gamma]}.
\end{align*}
The process to show that $r_{g}(\gamma, \delta)$ is the best possible is similar to (i). This completes the proof of (ii).
\end{proof}

\section{\bf Landau-Type theorems for the  class  $\mathcal{B H}^{0}(\gamma, \delta)$}\label{Sec-3}

For a continuously differentiable function $f$, we denote $\Lambda_f$ and $\lambda_f$ by
\begin{align*}
	\Lambda_f=\max_{0\leq\theta\leq 2\pi}|f_z+e^{-2i\theta}f_{\bar{z}}|=|f_z|+|f_{\bar{z}}|
\end{align*}
and
\begin{align*}
	\lambda_f=\min_{0\leq\theta\leq 2\pi}|f_z+e^{-2i\theta}f_{\bar{z}}|=||f_z|-|f_{\bar{z}}||.
\end{align*}
Thus, it is easy to see that for sense-preserving harmonic mapping $f$, one has $J_f=\Lambda_f\lambda_f$. A harmonic mapping $f=h+\overline{g}$ defined on the unit disc $\mathbb{D}$ can be expressed as
\begin{align*}
	f\left(re^{i\theta}\right)=\sum_{n=0}^{\infty}a_nr^ne^{in\theta}+\sum_{n=1}^{\infty}\overline{b_n}r^ne^{-in\theta},\; 0\leq r<1,
\end{align*}
where
\begin{align*}
	h(z)=\sum_{n=0}^{\infty}a_nz^n\;\; \mbox{and}\;\; g(z)=\sum_{n=1}^{\infty}b_nz^n.
\end{align*}
The classical landau theorem (see \cite{Landau-1926}) assets that if $f$ is a holomorphic mapping with $f(0)=0=f'(0)-1$ and $|f(z)|<M$ for $z\in\mathbb{D}$, then $f$ is univalent in $\mathbb{D}_{r_0}$, and $f(\mathbb{D}_{r_0})$ contains a disc $\mathbb{D}_{\sigma_0}$, where
\begin{align*}
	r_0=\frac{1}{M+\sqrt{M^2-1}}\; \mbox{and}\; \sigma_0=Mr_0^2.
\end{align*}
The quantities $r_0$ and $\sigma_0$ cannot be improved, the extremal function is
\begin{align*}
	f_0(z)=Mz\left(\frac{1-Mz}{M-z}\right).
\end{align*}
However, for holomorphic functions $f$ on the unit disc $\mathbb{D}$ with $f'(0)=1$, there is the Bloch theorem (see \cite{Bloch-Les-1925}) which asserts the existence of a positive constant $b$ such that $f(\mathbb{D})$ contains a schlicht disc of radius $b$. A disc $D$ is said to be schlicht disc if there exists a region $\Omega$ in the unit disc $\mathbb{D}$ such that $f$ is univalent on $\Omega$ and $f(\Omega)=D$. Let $b(f)$ denote the least upper bound of the radii of all schlicht discs that $f$ acrries and $\mathcal{F}$ denotes the set of all holomorphic functions defined on $\overline{\mathbb{D}}:=\{z :|z|\leq 1\}$ satisfying $|f'(0)|=1$, then the Bloch constant is the number defined by
\begin{align*}
	\beta(\mathcal{F})=\inf\{b(f) : f\in\mathcal{F}\}.
\end{align*}
If one considers the function $f(z)=z$, then clearly $\beta(\mathcal{F})\leq 1$. While the exact value of $\beta(\mathcal{F})$ is still unknown, better estimates have been established. In $1996$, Chen and Gauthier \cite{Chen-Gauthier-1996} proved that $\beta(\mathcal{F})$ lies within the interval $[0.4330, 0.4719]$, \textit{i.e.}, $0.4330\leq \beta(\mathcal{F})\leq 0.4719$.\vspace{1.2mm}

In $2000$, Chen \emph{et al.} \cite{Chen-Gauthier-Hengartner-2000} established the following two versions of Landau-Bloch type theorems for bounded harmonic mapping on the disc under a suitable restriction.
\begin{thm}
	Let $f$ be a harmonic mapping of the unit disc $\mathbb{D}$ such that $f(0)=0$, $f_{\bar{z}}(0)=0$, $f_z(0)=1$, and $|f(z)|<M$ for $z\in\mathbb{D}$. Then, $f$ is univalent on a disc $\mathbb{D}_{\rho_0}$ with
	\begin{align*}
		\rho_0=\frac{\pi^2}{16mM},
	\end{align*}
	and $f(\mathbb{D}_{\rho_0})$ contains a schlicht disc $\mathbb{D}_{R_0}$ with
	\begin{align*}
		R_0=\frac{\rho_0}{2}=\frac{\pi^2}{32mM},
	\end{align*}
	where $m\approx 6.85$ is the minimum of the function $\left(3-r^2\right)/(r(1-r^2))$ for $0<r<1$.
\end{thm}
\begin{thm}
	Let $f$ be a harmonic mapping of the unit disc $\mathbb{D}$ such that $f(0)=0$, $\lambda_f(0)=1$ and $\Lambda_f(z)\leq \Lambda$ for $z\in\mathbb{D}$. then, $f$ is univalent on a disc $\mathbb{D}_{\rho_0}$ with
	\begin{align*}
		\rho_0=\frac{\pi}{4(1+\Lambda)},
	\end{align*}
	and $f(\mathbb{D}_{\rho_0})$ contains a schlicht disc $\mathbb{D}_{R_0}$ with
	\begin{align*}
		R_0=\frac{\rho_0}{2}=\frac{\pi}{8(1+\Lambda)}.
	\end{align*}
\end{thm}
Theorems C and D are not sharp. In $2006$, better estimates for Theorems C and D were given by Grigoryan \cite{Grigoryan-CV-2006}. Specifically, Grigoryan proved the following lemma.
\begin{lemA}
	Assume that $f=h+\overline{g}$ with $h$ and $g$ analytic in the unit disc $\mathbb{D}$ and $h(z)=\sum_{n=1}^{\infty}a_nz^n$ and $g(z)=\sum_{n=1}^{\infty}b_nz^n$ for $z\in\mathbb{D}$.
	\begin{enumerate}
		\item[(a)] If $|f(z)|<M$ for $z\in\mathbb{D}$, then
		\begin{align*}
			|a_n|, \; |b_n|\leq M,\; n=1, 2, \ldots
		\end{align*}
		\item[(b)] If $\Lambda_f\leq \Lambda$ for $z\in\mathbb{D}$, then
		\begin{align*}
			|a_n|+|b_n|\leq \frac{\Lambda}{n},\; n=1, 2, \ldots
		\end{align*}
	\end{enumerate}
\end{lemA}
\begin{lemB}(The Schwarz Lemma).
	Let $f$ be a harmonic mapping of the unit disk $\mathbb{D}$
	such that $f(0) = 0$ and $f(\mathbb{D}) \subset \mathbb{D}$. Then
	\begin{align}\label{Eq-1.1}
		\Lambda_f(0)\leq\frac{4}{\pi},
	\end{align}
	\begin{align}
		\Lambda_f(z)\leq\frac{8}{\pi(1-|z|^2)},\;\mbox{for}\;z\in\mathbb{D},
	\end{align}
	\begin{align}
		|f(z)|\leq\frac{4}{\pi}\arctan|z|\leq\frac{4}{\pi}|z|,\;\mbox{for}\;z\in\mathbb{D},
	\end{align}
\end{lemB}
The following result has been proved by Grigoyan using Lemma A, which improved the estimates in Theorems A and B.
\begin{thm}
	Let $f$ be harmonic mapping of the unit disc $\mathbb{D}$ such that $f(0)=0$, $J_f(0)=1$ and $|f(z)|<M$ for $z\in\mathbb{D}$. Then, $f$ is univalent on a disc $\mathbb{D}_{\rho_1}$ with
	\begin{align*}
		\rho_1=1-\frac{2\sqrt{2}M}{\sqrt{\pi+8M^2}}
	\end{align*}
	and $f(\mathbb{D}_{\rho_1})$ contains a schlicht disc $\mathbb{D}_{R_1}$ with
	\begin{align*}
		R_1=\frac{\pi}{4M}+4M-4M\sqrt{1+\frac{\pi}{8M^2}}.
	\end{align*}
\end{thm}
\begin{thm}
	Let $f$ be harmonic mapping of the unit disc $\mathbb{D}$ such that $f(0)=0$, $\lambda_f(0)=1$ and $\Lambda_f(z)<\Lambda$ for all $z\in\mathbb{D}$. Then, $f$ is univalent on a disc $\mathbb{D}_{\rho_2}$ with
	\begin{align*}
		\rho_2=\rho_2(\Lambda)=\frac{1}{1+\Lambda}
	\end{align*}
	and $f(\mathbb{D}_{\rho_2})$ contains a schlicht disc $\mathbb{D}_{R_2}$ with
	\begin{align*}
		R_2(\Lambda)=1-\Lambda\ln\left(1+\frac{1}{\Lambda}\right).
	\end{align*}
\end{thm}
Our next result provides a Landau-type theorem for functions $f$ in the subclass $\mathcal{B H}^{0}(\gamma, \delta)$.
\begin{theorem}
	Let $f\in\mathcal{B H}^{0}(\gamma, \delta)$, $\delta\geq\gamma\geq0$ be a harmonic mapping on the unit disc $\mathbb{D}$ such that $f(0)=0, J_f(0)=1$ and $|f(z)|<M$ for $z\in\mathbb{D}$. Then, $f$ is univalent on a disc $\mathbb{D}_{\rho_1}$, where $\rho_1$ is the root of the equation
	\begin{align*}
		\frac{\pi}{4M} - \frac{4\ln(1-r)}{r^2} - \frac{4}{r} - \frac{8\gamma}{r(\delta-\gamma)} \Phi\left(r, 1, \frac{2\gamma}{\delta-\gamma} + 2\right)=0,
	\end{align*}
	and $\mathbb{D}_{\rho_1}$ contains a schlicht disc with
	\begin{align*}
		R_1=\frac{\pi}{4M}\rho_1 + \frac{4\ln(1-\rho_1)}{\rho_1} + 4 + 2\rho_1 + 4\rho_1^2 \Phi\left(\rho_1, 1, \frac{2\gamma}{\delta-\gamma} + 2\right).
	\end{align*}
	This result is sharp, with an extremal function given by
	\begin{align}\label{Eq-4.5A}
		F_1(z)=\frac{\pi}{4M}z + \frac{4\ln(1-z)}{z} + 4 + 2z + 4z^2 \Phi\left(z, 1, \frac{2\gamma}{\delta-\gamma} + 2\right).
	\end{align}
\end{theorem}
\begin{proof}
	Fix $r\in(0,1)$ and $z_1\neq z_2$. We estimate the modulus of the expression
	\begin{align}\label{Eq-2.2AB}
		f(z_1)-f(z_2)&=\int_{[z_1,z_2]}f_z(z)dz+f_{\overline{z}}(z)d\overline{z}\\&=\int_{[z_1,z_2]}h'(z)dz+{\overline{g'}}(z)d\overline{z}\nonumber
	\end{align}
	where $\gamma=[z_1,z_2], \gamma(t)=z_1(1-t)+tz_2$. By the  Schwarz lemma, we have
	\begin{align}\label{Eq-2.3AB}
		\lambda_f(0)=\frac{1}{\Lambda_f(0)}\geq\frac{\pi}{4M}.
	\end{align}
	Combining \eqref{Eq-2.2AB} and \eqref{Eq-2.3AB} via the triangle inequality, we obtain
	\begin{align*}
		|f(z_1)-f(z_2)|&\geq\bigg|\int_{[z_1,z_2]}h'(z)dz+{\overline{g'}}(z)d\overline{z}\bigg|\\&\geq\bigg|\int_{[z_1,z_2]}h'(0)dz+{\overline{g'}}(0)d\overline{z}\bigg|\\&\quad-\bigg|\int_{[z_1,z_2]}(h'(z)-h'(0))dz+({\overline{g'}}(z)-{\overline{g'}}(0))d\overline{z}\bigg|\\&\geq\lambda_f(0)|z_1-z_2|-|h(z_2)-h(z_1)-h'(0)(z_1-z_2)|\\&\quad-|g(z_2)-g(z_1)-g'(0)(z_1-z_2)|\\&\geq|z_1-z_2|\left(\frac{\pi}{4M}-\sum_{n=2}^{\infty}(|a_n|+|b_n|)nr^{n-1}\right).
	\end{align*}
	Then by Lemma A, using the coefficient estimate, we obtain
	\begin{align}\label{Eq-4.8}
		|f(z_1)-f(z_2)|&\geq|z_1-z_2|\left(\frac{\pi}{4M}-\sum_{n=2}^{\infty}\left(\dfrac{4(\delta+\gamma)}{(n+1)[n(\delta-\gamma)+2 \gamma]}\right)nr^{n-1}\right)\nonumber\\&\geq|z_1-z_2|\left(\frac{\pi}{4M} - \frac{4\ln(1-r)}{r^2} - \frac{4}{r} - \frac{8\gamma}{r(\delta-\gamma)} \Phi\left(r, 1, \frac{2\gamma}{\delta-\gamma} + 2\right)\right).
	\end{align}
	Let
	\begin{align*}
		\mathcal{J}(r)=\frac{\pi}{4M} - \frac{4\ln(1-r)}{r^2} - \frac{4}{r} - \frac{8\gamma}{r(\delta-\gamma)} \Phi\left(r, 1, \frac{2\gamma}{\delta-\gamma} + 2\right),
	\end{align*}
	we have to check monotonecity. Therefore,
	\begin{align*}
		\mathcal{J}(0)=\frac{\pi}{4M} + 2 + \frac{8\gamma}{\delta - 3\gamma}>0\;\mbox{and}\;\lim_{r\rightarrow1^{-}}\mathcal{J}(r)<0,\;\mbox{if}\;\gamma\neq3\delta.
	\end{align*}
	\begin{table}[h]
		\centering
		\renewcommand{\arraystretch}{1.5}
		\begin{tabular}{|l|c|c|c|c|c|}
			\hline
			\textbf{Point} & \textbf{1} & \textbf{2} & \textbf{3} & \textbf{4} & \textbf{5} \\ \hline
			$\delta$       & 1.0        & 5.0        & 4.0        & 3.5        & 3.2        \\ \hline
			$\gamma$       & 0.0        & 1.0        & 1.0        & 1.0        & 1.0        \\ \hline
			$\eta_1(0)$    & \textbf{2.785} & \textbf{6.785} & \textbf{10.785} & \textbf{18.785} & \textbf{42.785} \\ \hline
		\end{tabular}
		\caption{Positive values of $\eta_1(0)$ for $M=1$ and $\delta > 3\gamma$.}
	\end{table}
	\begin{table}[h]
		\centering
		\renewcommand{\arraystretch}{1.5}
		\begin{tabular}{|l|c|c|c|c|c|}
			\hline
			\textbf{Point} & \textbf{1} & \textbf{2} & \textbf{3} & \textbf{4} & \textbf{5} \\ \hline
			$\delta$       & 1.0        & 5.0        & 4.0        & 3.5        & 3.2        \\ \hline
			$\gamma$       & 0.0        & 1.0        & 1.0        & 1.0        & 1.0        \\ \hline
			$\eta_1(0)$    & \textbf{2.393} & \textbf{6.393} & \textbf{10.393} & \textbf{18.393} & \textbf{42.393} \\ \hline
		\end{tabular}
		\caption{Positive values of $\eta_1(0)$ for $M=2$ and $\delta > 3\gamma$.}
	\end{table}
	
	\begin{table}[h]
		\centering
		\renewcommand{\arraystretch}{1.5}
		\begin{tabular}{|l|c|c|c|c|c|}
			\hline
			\textbf{Point} & \textbf{1} & \textbf{2} & \textbf{3} & \textbf{4} & \textbf{5} \\ \hline
			$\delta$       & 1.0        & 5.0        & 4.0        & 3.5        & 3.2        \\ \hline
			$\gamma$       & 0.0        & 1.0        & 1.0        & 1.0        & 1.0        \\ \hline
			$\eta_1(0)$    & \textbf{2.262} & \textbf{6.262} & \textbf{10.262} & \textbf{18.262} & \textbf{42.262} \\ \hline
		\end{tabular}
		\caption{Positive values of $\eta_1(0)$ for $M=3$ and $\delta > 3\gamma$.}
	\end{table}
	
	\begin{table}[h]
		\centering
		\renewcommand{\arraystretch}{1.5}
		\begin{tabular}{|l|c|c|c|c|c|}
			\hline
			\textbf{Point} & \textbf{1} & \textbf{2} & \textbf{3} & \textbf{4} & \textbf{5} \\ \hline
			$\delta$       & 1.0        & 5.0        & 4.0        & 3.5        & 3.2        \\ \hline
			$\gamma$       & 0.0        & 1.0        & 1.0        & 1.0        & 1.0        \\ \hline
			$\eta_1(0)$    & \textbf{2.196} & \textbf{6.196} & \textbf{10.196} & \textbf{18.196} & \textbf{42.196} \\ \hline
		\end{tabular}
		\caption{Positive values of $\eta_1(0)$ for $M=4$ and $\delta > 3\gamma$.}
	\end{table}
	Again,
	\begin{align*}
		\mathcal{J}'(r)=\frac{8 \ln(1-r)}{r^3} + \frac{4}{r^2} + \frac{4(\delta-3\gamma)}{r^2(1-r)(\delta-\gamma)} + \frac{8\gamma(3\delta-\gamma)}{r^2(\delta-\gamma)^2} \Phi\left(r, 1, \frac{2\delta}{\delta-\gamma}\right).
 	\end{align*}
	This implies that $\mathcal{J}(r)$ is a strictly decreasing function on $(0, 1)$. Since $\mathcal{J}(r)$ is continuous, the Intermediate Value Theorem ensures the existence of a root, say $\rho_1 \in (0, 1)$.
	From equation \eqref{Eq-4.8}, it follows that $|f(z_1) - f(z_2)| > 0$ for all $|z_1|, |z_2| < \rho_1$ with $z_1 \neq z_2$. This establishes the univalence of $f$ in the disc $\mathbb{D}_{\rho_1}$.
	
	Let $|z|=\rho_1$, we have
	\begin{align*}
		|f(z)|&\geq|a_1z+b_1\overline{z}|-\sum_{n=2}^{\infty}(|a_n|+|b_n|)|z|^{n}\\&\geq\frac{\pi}{4M}\rho_1-\sum_{n=2}^{\infty}\left(\dfrac{4(\delta+\gamma)}{(n+1)[n(\delta-\gamma)+2 \gamma]}\right)\rho_1^n\\&\geq\frac{\pi}{4M}\rho_1 + \frac{4\ln(1-\rho_1)}{\rho_1} + 4 + 2\rho_1 + 4\rho_1^2 \Phi\left(\rho_1, 1, \frac{2\gamma}{\delta-\gamma} + 2\right):=R_1.
	\end{align*}
	To show that the univalent radius $\rho_1$ is sharp, we need to prove that $F_1(z)$	is not univalent in $\mathbb{D}_r$ for each $r \in (\rho_1, 1]$. In fact, considering the real differentiable
	function
	\begin{align}\label{Eq-4.9A}
		h_0(x)=\frac{\pi}{4M}x + \frac{4\ln(1-x)}{x} + 4 + 2x + 4x^2 \Phi\left(x, 1, \frac{2\gamma}{\delta-\gamma} + 2\right), \quad x \in [0, 1].
	\end{align}
	Since the continuous function
	\begin{align*}
		h_0'(x) = \frac{\pi}{4M} - \frac{4\ln(1-x)}{x^2} - \frac{4}{x} - 2 - \frac{8\gamma x}{\delta-\gamma} \Phi\left(x, 1, \frac{2\gamma}{\delta-\gamma} + 2\right)
	\end{align*}
	is strictly decreasing on $[0, 1]$ and $h_0'(\rho_1) = 0$, it follows that $h_0'(x) > 0$ for $x \in [0, \rho_1)$ and $h_0'(x) < 0$ for $x \in (\rho_1, 1]$. Consequently, $h_0(x)$ is strictly increasing on $[0, \rho_1]$ and strictly decreasing on $[\rho_1, 1]$. Since $h_0(0) = 0$, there is a unique real $r_1\in (\rho_1, 1]$
	such that $h_0(r_1) = 0$ if $\lim_{x\rightarrow1^{-}}h_0(x) \leq 0$, and
	\begin{align}\label{Eq-4.10}
		R_1 =\frac{\pi}{4M}\rho_1 + \frac{4\ln(1-\rho_1)}{\rho_1} + 4 + 2\rho_1 + 4\rho_1^2 \Phi\left(\rho_1, 1, \frac{2\gamma}{\delta-\gamma} + 2\right)> h_0(0) = 0.
	\end{align}
	For every fixed $r \in (\rho_1, 1]$, set $x_1 = \rho_1 + \varepsilon$, where
	\[ \varepsilon = \begin{cases}
		\min\left\{ \dfrac{r-\rho_1}{2}, \dfrac{r_1-\rho_1}{2} \right\}, & \text{if } f_0(1) \leq 0,\vspace{2mm} \\
		\dfrac{r-\rho_1}{2}, & \text{if } f_0(1) > 0.
	\end{cases} \]
	by the mean value theorem, there is a unique $\delta \in (0, \rho_1)$ such that $x_2 :=
	\rho_1 - \delta \in (0, \rho_1)$ and $h_0(x_1) = h_0(x_2)$.
	Let $z_1 = x_1$ and $z_2 = x_2$. Then $z_1, z_2 \in\mathbb{D}_r$ with $z_1 \neq z_2$. A simple
	computation leads to
	\[F_1(z_1) = F_1(x_1) = h_0(x_1) = h_0(x_2) = F_1(z_2).\]
	Hence $F_1$ is not univalent in the disk $\mathbb{D}_r$ for each $r \in (\rho_1, 1]$, and the univalent
	radius $\rho_1$ is sharp.\vspace*{2mm}

	Finally, note that $F_1(0) = 0$ and picking up $z' = \rho_1\in \partial \mathbb{D}_{\rho_1}$, by \eqref{Eq-4.5A},
	\eqref{Eq-4.9A} and \eqref{Eq-4.10}, we have
	\begin{align*}
		|F_1(z') - F_1(0)| = |F_1(\rho_1)| = |h_0(\rho_1)| = h_0(\rho_1) = R_1.
	\end{align*}
	Hence, the covering radius $R_1$ is also sharp.
	This completes the proof.
\end{proof}

\vskip .20in
\begin{center}{\sc acknowledgments}
\end{center}

\vskip.01in

The authors   wish to thank  the anonymous reviewers for thorough reading of the manuscript
and instructive suggestions on improving the  paper.
\vskip .20in
\begin{center}{\sc Conflict of Interests}
\end{center}
The authors declare that they have no conflict of interest  regarding the publication of this paper. All authors contributed equally to this work.

\vskip .20in
\begin{center}{\sc Data Availability Statement}
\end{center}
The authors declare that this research is purely theoretical and does not associate with any data.

\end{document}